\documentclass[12pt]{article}

\usepackage{amsfonts, amssymb, amsmath, amsthm}
\usepackage{enumitem}
\usepackage{mathrsfs}
\usepackage{float}
\usepackage{fancyhdr}
\usepackage{caption}
\usepackage{empheq}

\usepackage{tikz}
\usepackage[all]{xy}
\usepackage{tensor}
\usepackage{MnSymbol}
\usepackage{bbold}

\usetikzlibrary{arrows}
\usetikzlibrary{shapes.arrows}

\usepackage{color,hyperref}
    \catcode`\_=11\relax
    \newcommand\email[1]{\_email #1\q_nil}
    \def\_email#1@#2\q_nil{%
      \href{mailto:#1@#2}{{\emailfont #1\emailampersat #2}}
    }
    \newcommand\emailfont{\sffamily}
    \newcommand\emailampersat{{\color{red}\small@}}
    \catcode`\_=8\relax  
    
\theoremstyle{plain}
\newtheorem{thm}{Theorem}[section]
\newtheorem{lemma}[thm]{Lemma}

\makeatletter
\newcommand{\superimpose}[2]{%
  {\ooalign{$#1\@firstoftwo#2$\cr\hfil$#1\@secondoftwo#2$\hfil\cr}}}
\makeatother

\newcommand{\be}{\begin{equation}} % begin equation
\newcommand{\ee}{\end{equation}}

\newcommand{\sa}{\Sigma}
\newcommand{\M}{\mathbf{Meas}}

\newcommand{\mcS}{P_{Y|X}}

\newcommand{\B}{\mathcal{B}}
\newcommand{\G}{\mathcal{G}}   % Giry monad

\newcommand{\T}{\mathcal{G}}   % Probability monad using Cvx_w(I^,I)

\def\@normalsize{\@setsize\normalsize{14.5pt}\xiipt\@xiipt
\abovedisplayskip 12\p@ plus3\p@ minus7\p@
\belowdisplayskip \abovedisplayskip
\abovedisplayshortskip  \z@ plus3\p@
\belowdisplayshortskip  6.5\p@ plus3.5\p@ minus3\p@
\let\@listi\@listI}
\title{Bayesian Inference \\ using the \\ Symmetric Monoidal Closed Category Structure}
\date{}
\author{Kirk Sturtz}

\fancypagestyle{plain}{% Redefine ``plain'' style for chapter boundaries  
\fancyhead{Draft} % clear all header and footer fields  
\fancyfoot[C]{\thepage} % except the center  
  
}

\begin{document}
 \maketitle
 
\thispagestyle{empty}

\begin{abstract}
Using the symmetric monoidal closed category structure of the category of measurable spaces, in conjunction with the Giry monad which we show is a strong monad, we analyze Bayesian inference maps and their construction in relation to the tensor product probability.  This perspective permits the inference maps to be seen as a pullback construction.
\end{abstract}

\section{Introduction} 
A theory for constructing Bayesian inference maps is developed by exploiting the symmetric monoidal closed  category (SMCC) structure of the category of measurable spaces, $\M$.  Using this property the inference maps can be constructed as pullbacks.

While the construction of these maps requires the SMCC structure of $\M$, the  Bayesian inference problem itself is most naturally characterized within the Kleisi category of the Giry monad, $K(\G)$, where $\G$ denote the Giry monad on the category of measurable spaces, $\M$.\cite{Giry}  In this introduction we provide a general overview of this perspective which permits one to quickly understand precisely what the \emph{Bayesian inference problem} entails.

A Bayesian model is a pair $(P_X, P_{Y|X})$ consisting of a  probability measure \mbox{$P_X \in \G(X)$},  and a (regular) conditional probability measure $P_{Y|X}: X \rightarrow \G(Y)$.\footnote{A regular conditional probability is also commonly referred to as a kernel.} Thus at  each point \mbox{$x \in X$}, $P_{Y|X}(\cdot | x)$ is a probability measure on $Y$, and this  conditional probability $\mcS$ in turn  determines the conditional probability $P_{X \times Y |X}$ defined as the probability measure on $X \times Y$ at every $x \in X$ as the composite
  \begin{equation}   \nonumber
 \begin{tikzpicture}[baseline=(current bounding box.center)]
 	\node	(1)	at	(-6,-4)		         {$1$};
	\node	(Y)	at	(1,-4)	               {$Y$};
	\node	(X)	at	(-3,-4)	               {$X$};
	\node        (XY)  at      (-3,-2)               {$X \times Y$};
	
	\draw[->,left] (1) to node [xshift=-7pt,yshift=6pt] {$P_{X \times Y|X}( \cdot | x)$} (XY);
	
	\draw[->,above] (1) to node {$\ulcorner x \urcorner$} (X);
	\draw[->,above] (Y) to node [xshift=5pt]{$\delta_{\Gamma_x}$} (XY);
	\draw[->,right,dashed] (X) to node  [xshift=-0pt,yshift=0pt]{$P_{X \times Y | X}$} (XY);
	\draw[->,above] ([yshift=0pt] X.east) to node {$P_{Y|X}$} ([yshift=0pt] Y.west);
%	\draw[->, below,dashed] ([yshift=-3pt] Y.west) to node {$P_{X|Y}$} ([yshift=-3pt] X.east);

 \end{tikzpicture}
 \end{equation}
giving the push forward probability measure  $P_{X \times Y |X}(\cdot | x) = P_{Y|X}(\Gamma_x^{-1}(\cdot) | x)$, where 
  \begin{equation}   \nonumber
 \begin{tikzpicture}[baseline=(current bounding box.center)]

	\node	(Y)	at	(0,0)	               {$Y$};
	\node        (XY)  at      (3,0)               {$X \times Y$};
		\node	(y)	at	(0,-.8)	               {$y$};
	\node        (xy)  at      (3,-.8)               {$(x,y)$};
	
	\draw[->,above] (Y) to node {$\Gamma_x$} (XY);
	\draw[|->] (y) to node {} (xy);
%	\draw[->, below,dashed] ([yshift=-3pt] Y.west) to node {$P_{X|Y}$} ([yshift=-3pt] X.east);

 \end{tikzpicture}
 \end{equation}
is the constant graph map. On the other hand, given $P_{X \times Y|X}$ we have
 \be \nonumber
 \begin{array}{lcl}
 P_{X|Y}(A | y) &=& (\delta_{\pi_X} \star P_{X \times Y|Y})(A | y) \\
 &=& P_{X \times Y|Y}(\pi_X^{-1}(A) | y)
 \end{array}
 \ee
 where $\pi_X: X \times Y \rightarrow X$ is the coordinate projection map.
These two processes are inverse to each other, and hence knowledge of either $P_{X|Y}$ or $P_{X\times Y|Y}$ solves the inference problem.

Given a Bayesian model $(P_X, P_{Y|X})$ we can construct the  $K(\G)$-diagram
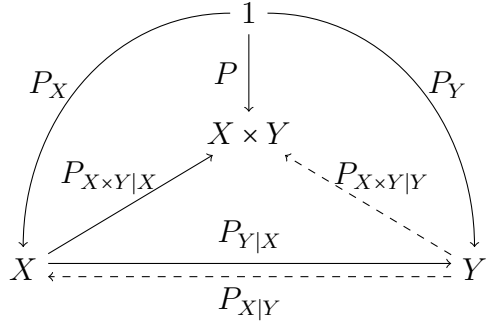
\begin{figure}[H]
  \begin{equation}   \nonumber
 \begin{tikzpicture}[baseline=(current bounding box.center)]
 	\node	(1)	at	(0,-.6)		         {$1$};
	\node	(Y)	at	(3,-4)	               {$Y$};
	\node	(X)	at	(-3,-4)	               {$X$};
	\node        (XY)  at      (0,-2.2)               {$X \times Y$};
	
	\draw[->,out=0,in=90,right] (1) to node  {$P_Y$} (Y);
	\draw[->,out=180,in=90,left] (1) to node  {$P_X$} (X);
	\draw[->,left] (1) to node {$P$} (XY);
	\draw[->,above,dashed] (Y) to node [xshift=5pt]{$P_{X \times Y | Y}$} (XY);
	\draw[->,above] (X) to node  [xshift=-8pt,yshift=0pt]{$P_{X \times Y | X}$} (XY);
	\draw[->,above] ([yshift=2pt] X.east) to node {$P_{Y|X}$} ([yshift=2pt] Y.west);
	\draw[->, below,dashed] ([yshift=-3pt] Y.west) to node {$P_{X|Y}$} ([yshift=-3pt] X.east);

 \end{tikzpicture}
 \end{equation}
 \caption{The Bayesian inference problem.}
\label{BayesDiagram}
 \end{figure}
 \noindent
where by definition $P = P_{X \times Y|X} \star P_X$, and  $P_Y = P\pi_Y^{-1}$ denotes the marginal probability measure on $Y$ given $P$.  
An inference map for the Bayesian model is a conditional probability $P_{X \times Y|Y}$ such that one can use $P_Y$ and $P_{X\times Y|Y}$, in lieu of $P_X$ and $P_{X \times Y|X}$, to determine the \emph{same} joint probability measure $P$ on $X \times Y$,
\be \nonumber
P_{X \times Y|X} \star P_X = P = P_{X \times Y|Y} \star P_Y  
\ee
which is Bayes equation.
Using the constant graph maps Bayes equation  can be written  
\be   \nonumber
 \int_X P_{Y|X}(\Gamma_x^{-1}(\zeta) | x) \, dP_X =  \int_Y P_{X|Y}(\Gamma_y^{-1}(\zeta) | y) \, dP_Y \quad \forall \zeta \in \sa_{X \times Y} 
\ee
 
Provided we restrict $\M$ to the category of Standard Spaces or Polish Spaces, the existence of inference maps is well known.\cite{Faden,Winter}  It is possible to extend this class by using $\mathcal{D}$-kernels rather than kernels which are \emph{regular} conditional probabilities which we have been, and will continue, to call conditional probabilities.  Using $\mathcal{D}$-kernels requires an awkward category and hence we assume Standard Spaces.  This issue  is discussed further in the appendix.  Our choice of Standard Spaces is philosophical based upon the idea probability theory should be independent of any topological properties. A good source covering the essential aspects of Standard Spaces is \cite[Chapter 2]{Gray}. Faden\cite[Proposition 5]{Faden} shows that almost pre-standard spaces suffice.

The objective of this paper is to serve as a stepping stone to develop better computational methods for inference.  Towards that end, we believe it is necessary to exploit the SMCC structure of $\M$,  in conjunction with the strong monad structure of $\G$,  to develop better computational schemes for inference.   The basic structural maps, which are natural transformations,  discussed herein provide the basis for such developments.    
 
 \section{Some Background and Notation} 
 
 A measurable function is called a map or a random variable. The $\sigma$-algebra associated with the product space $X \times Y$ is the product $\sigma$-algebra.  
 %We emphasize that our use of the term conditional probability refers to a kernel or a regular conditional probability.
If a probability on any space assumes only one of the two values $\{0,1\}$ we call the  probability \emph{deterministic}.
Every map $f:X \rightarrow Y$ induces a deterministic probability, defined for all $B \in \sa_Y$ and $x \in X$ by
   \begin{equation}   \nonumber
 \begin{tikzpicture}[baseline=(current bounding box.center)]
 	\node	(X)	at	(0,0)		         {$X$};
	\node	(Xf)	at	(3,0)	               {$Y$};
	\node       (c)    at      (8,0)      {$\delta_{f}(B | x) = \left\{ \begin{array}{ll} 1 & \textrm{ iff }f(x) \in B \\ 0 & \textrm{ otherwise } \end{array} \right..$};
	\draw[->,below] (X) to node {$\delta_{f}$} (Xf);

 \end{tikzpicture}
 \end{equation}
Using the graph  of $f$, $\Gamma_f: X \rightarrow X \times Y$ which maps $x \mapsto (x,f(x))$ gives  the deterministic measure
   \begin{equation}   \nonumber
 \begin{tikzpicture}[baseline=(current bounding box.center)]
 	\node	(X)	at	(0,0)		         {$X$};
	\node	(Xf)	at	(3,0)	               {$X \times Y$};
	\node       (c)    at      (8,0)      {$\delta_{\Gamma_{f}}(\zeta | x) = \left\{ \begin{array}{ll} 1 & \textrm{ iff }(x,f(x))  \in \zeta  \\ 0 & \textrm{ otherwise } \end{array} \right.$};
	\draw[->,below] (X) to node {$\delta_{\Gamma_f}$} (Xf);

 \end{tikzpicture}
 \end{equation}
 which plays an essential role in the theory.
 
 If a Bayesian model $(P_X, P_{Y|X})$ has a deterministic probability $P_{Y|X}$ then we say we have a \emph{deterministic} Bayesian model, and write it as 
 $(P_X, \delta_f)$ where $f:X \rightarrow Y$ is the map giving rise to the deterministic probability.\footnote{In a countably generated measurable space very deterministic measure arises from a measurable function  \cite{Culbertson}.  Consequently, under the assumption of Standard Spaces, every deterministic measure arises from a map $f$ and vice versa.}  If $P_{Y|X}$ is not deterministic then we say the model $(P_X, P_{Y|X})$ is a nondeterministic Bayesian model.

The tensor product monad  $(\M,\otimes,1)$ structure on $\M$, defined on the objects by
 \begin{equation}   \nonumber
 \begin{tikzpicture}[baseline=(current bounding box.center)]
 	\node	(X)	at	(0,0)		         {$\M \times \M$};
	\node	(Y)	at	(4,0)	               {$\M$};
	 	\node	(x)	at	(0,-.8)		         {$X \times Y$};
	\node	(y)	at	(4,-.8)	               {$X \otimes Y$};

	\draw[->,above] (X) to node {$\otimes$} (Y);
	\draw[|->,above] (x) to node {} (y);

 \end{tikzpicture}
 \end{equation}  
 where $X \otimes Y$ is the set $X \times Y$ with the $\sigma$-algebra generated by the family of constant graph maps, 
 $\Gamma_x$ and $\Gamma_y$,
 \begin{equation}   \nonumber
 \begin{tikzpicture}[baseline=(current bounding box.center)]
 	\node	(XY)	at	(0,0)		         {$X \otimes Y$};
	\node	(X)	at	(-2,0)	               {$X$};
	\node        (Y)  at      (2,0)                  {$Y$};
	 \node	(x)	at	(-2,-.8)		         {$x$};
	\node	(y)	at	(2,-.8)	               {$y$};
	\node        (xy)  at     (0,-.8)            {$(x,y)$};

	\draw[->,above] (X) to node {$\Gamma_y$} (XY);
	\draw[->,above] (Y) to node {$\Gamma_x$} (XY);
	
		\draw[|->] (x) to node {} (xy);
	\draw[|->,above] (y) to node {} (xy);

 \end{tikzpicture}
 \end{equation}  
makes $\M$ a symmetric monoidal closed category.\footnote{Further details concerning the SMCC structure of $\M$ can be found in \cite{Sturtz}. The necessary proofs follow readily from the definition of the final $\sigma$-algebra on the set $X \times Y$ via the constant graph maps. }  Thus  the evaluation maps

\be \nonumber 
 \begin{tikzpicture}[baseline=(current bounding box.center)]
        \node          (X)    at      (-3,0)             {$X$};
 	\node	(YX)	at	(3,0)              {$Y^X$};
	\node	(XY)	at	(0,0)	               {$ X \otimes Y^X$};
	\node	(Y)	at	(0,-1.8)               {$Y$};

         \draw[->,above] (X) to node {$\Gamma_{\ulcorner f \urcorner}$} (XY);
         \draw[->,left] (X) to node [xshift=-7pt] {$f$} (Y);
	\draw[->, above] (YX) to node  {$\Gamma_{x}$} (XY);
	\draw[->,right] (YX) to node [xshift=5pt,yshift=0pt] {$ev_x$} (Y);
	\draw[->,right] (XY) to node {$ev_{X,Y}$} (Y);

 \end{tikzpicture}
 \ee
\noindent 
where the function spaces $Y^X$ are defined such that the point evaluation maps are all measurable.  Recall, the product $\sigma$-algebra is a sub $\sigma$-algebra of the tensor $\sigma$-algebra, and, typically we employ the tensor product $\sigma$-algebra only when we require the use of the evaluation maps. Without the finer $\sigma$-algebra the tensor product $\sigma$-algebra provides, the evaluation maps are not  measurable.

\vspace{.1in}

\section{Some structural maps in $(\M,\otimes,1)$}
The monoidal closed structure of $\M$, in conjunction with the strong monad structure of $\G$,  provides several   natural transformation which are indispensable tools.  These basic  ``structural maps'' are due to A. Kock \cite{Kock1,Kock2}.   As we need to show $\G$ is a strong monad it is necessary to prove 

\vspace{.1in}
\begin{thm}
There is a natural transformation between the two functors
\be \nonumber
\cdot \otimes \G(\cdot) \, \G(\cdot \times \cdot): \M \times \M \rightarrow \M
\ee
defined at component $(X,Y)$ by
\begin{equation}   \nonumber
 \begin{tikzpicture}[baseline=(current bounding box.center)]
         \node       (XY)    at      (0,0)          {$X \otimes \G(Y)$};
 	\node	(XY2)	at	(5,0)     {$\G(X \times Y)$};
	\node       (y)      at      (0,-.8)          {$(x,Q)$};
 	\node	(xy)	at	(5,-.8)       {$Q\Gamma_x^{-1}$};

	\draw[->,above] (XY) to node {$\tau''_{X,Y}$} (XY2);
	\draw[|->] (y) to node {} (xy);

 \end{tikzpicture}.
 \end{equation}
 \end{thm}
 \begin{proof}
To verify $\tau''$  is measurable consider the diagram
\begin{equation}   \nonumber
 \begin{tikzpicture}[baseline=(current bounding box.center)]
         \node       (GX)    at      (0,0)          {$X$};
 	\node	(Y)	at	(6,0)     {$\G(Y)$};
	\node        (GXY) at   (3,0)   {$X \otimes \G(Y)$};
	\node         (GXX) at  (3,-1.8)  {$\G(X \times Y)$};
	\node         (I)       at    (3,-4)  {$[0,1]$};

	\draw[->,above] (GX) to node {$\Gamma_Q$} (GXY);
	\draw[->,above] (Y) to node {$\Gamma_x$} (GXY);
	\draw[->,right] (GXY) to node {$\tau''$} (GXX);
	\draw[->,right] (GXX) to node [yshift=3pt]{$ev_{\zeta}$} (I);
	\draw[->,left] (GX) to node {} (I);
	\draw[->,right] (Y) to node {$ev_{\Gamma_x^{-1}(\zeta)}$} (I);
	
        \node       (GX2)    at      (11,0)          {$Q$};
% 	\node	(Y2)	at	(14,0)     {$y$};
	\node        (GXY2) at   (9,0)   {$(x,Q)$};
	\node         (GXX2) at  (9,-1.8)  {$Q\Gamma_x^{-1}$};
	\node         (I2)       at    (9,-4)  {$Q\Gamma_x^{-1}(\zeta)$};
	
	\draw[->,above] (GX2) to node {$\Gamma_x$} (GXY2);
%	\draw[->,above] (Y2) to node {$\Gamma_P$} (GXY2);
	\draw[->,left] (GXY2) to node {$\tau''_{X,Y}$} (GXX2);
	\draw[->,left] (GXX2) to node [yshift=3pt]{$ev_{\zeta}$} (I2);
	\draw[->,left] (GX2) to node {} (I2);
%	\draw[->,right] (Y2) to node {} (I2);

 \end{tikzpicture}.
 \end{equation}
For a fixed $Q \in \G(Y)$, the composite mapping $x \mapsto Q\Gamma_x^{-1}(\zeta)$ is well known to be \\
 \mbox{measurable\cite[Proposition 5.1.2, p156]{Cohn}}. On the other hand, for a fixed $x \in X$, the composite map on the right hand triangle in the above diagram is precisely the evaluation map $ev_{\Gamma_x^{-1}(\zeta)}$, and these maps generate the $\sigma$-algebra on $\G(Y)$.   Thus, for every $\zeta \in \sa_{X \times Y}$, 
\be \nonumber
\Gamma_x^{-1}({\tau''_{X,Y}}^{-1}(ev_{\zeta}^{-1}(U))) \in \G(Y) \quad \forall U \in \sa_{I}
\ee
Since these are all measurable in $\G(Y)$, and $X \otimes \G(Y)$ has the largest $\sigma$-algebra such that all the constant graph maps are measurable it follows the argument of $\Gamma_x^{-1}$ in the above equation is measurable.  Thus ${\tau''_{X,Y}}^{-1}(ev_{\zeta}^{-1}(U)) \in \sa_{X \otimes \G(Y)}$.
Now, using the fact $\G(X \times Y)$ is generated by the evaluation maps $ev_{\zeta}$, the measurability of $\tau''_{X,Y}$  follows.

\vspace{.1in}

To prove naturality in the first argument note that for $X \stackrel{f}{\longrightarrow} X'$ we obtain the $\M$ arrow $\G(X) \stackrel{\G(f)}{\longrightarrow} \G(X)'$ mapping  $P \mapsto Pf^{-1}$.  This gives  the commutative square
 \begin{equation} \nonumber
 \begin{tikzpicture}[baseline=(current bounding box.center)]
 	\node	(XpTY)	at	(0,2)		         {$X' \otimes \G(Y)$};
	\node	(TXpY)	at	(4.5,2)	                  {$\G(X' \otimes Y)$};
	\node	(XTY)	at	(0,0)	                  {$X \otimes \G(Y)$};
	\node        (TXY)        at       (4.5,0)                  {$\G(X \otimes Y)$};
	
	\node	(GXpY)	at	(9,2)	                  {$\G(X' \times Y)$};
	\node        (GXY)        at       (9,0)                  {$\G(X \times Y)$};

	\draw[->, right,auto] (XpTY) to node  {$\tau^{''}_{X',Y}$} (TXpY);
	\draw[->, above, auto] (XTY) to node {$f \otimes 1_{\G(Y)}$} (XpTY);
	\draw[->,below]  (XTY)  to node {$\tau^{''}_{X,Y}$}  (TXY);
	\draw[->,right] (TXY) to node {$\G(f \otimes 1_Y)$} (TXpY);
	
	\draw[->,below] (TXY) to node {$\G(id_{X \times Y})$} (GXY);
	\draw[->,above] (TXpY) to node {$\G(id_{X' \times Y})$} (GXpY);
	\draw[->,right] (GXY) to node {$\G(f \times 1_Y)$} (GXpY);
%	\draw (7.4,-1.1) node {$ev$};                    % auto placement of label sucks.	
 \end{tikzpicture}
 \end{equation}
 which maps elements according to
  \begin{equation} \nonumber
 \begin{tikzpicture}[baseline=(current bounding box.center)]
 	\node	(XpTY)	at	(0,2)		         {$(f(x), Q)$};
	\node	(TXpY)	at	(5,2)	                  {$Q \Gamma_{f(x)}^{-1}=Q \Gamma_x^{-1} (f \otimes 1_Y)^{-1} $};
	\node	(XTY)	at	(0,0)	                  {$(x,Q)$};
	\node        (TXY)        at       (5,0)                  {$Q \Gamma_x^{-1}$};
	\draw[|->, right,auto] (XpTY) to node  {} (TXpY);
	\draw[|->, above, auto] (XTY) to node {} (XpTY);
	\draw[|->,right,auto]  (XTY)  to node {}  (TXY);
	\draw[|->,above,auto] (TXY) to node {} (TXpY);

 \end{tikzpicture}
 \end{equation}
where equality in the upper right hand corner follows from the observation that 
\begin{equation}   \nonumber
 \begin{tikzpicture}[baseline=(current bounding box.center)]
          \node       (XYp)    at      (0,2)          {$X \times Y$};
 	\node	(XY2p)	at	(3,2)     {$X' \times Y$};

         \node       (XY)    at      (0,0)          {$X \otimes Y$};
 	\node	(XY2)	at	(3,0)     {$X' \otimes Y$};
	\node       (Y)      at      (0,-2)          {$Y$};

        \draw[->,above] (XYp) to node {$f \times 1_Y$} (XY2p);
        \draw[->,left] (XY) to node {$id_{X \times Y}$} (XYp);
        \draw[->,right] (XY2) to node {$id_{X' \times Y}$} (XY2p);
	\draw[->,above] (XY) to node {$f \otimes 1_Y$} (XY2);
	\draw[->,left] (Y) to node {$\Gamma_x$} (XY);
	\draw[->,right] (Y) to node [xshift=5pt]{$\Gamma_{f(x)}$} (XY2);

 \end{tikzpicture}.
 \end{equation}

\vspace{.1in}
To prove naturality in the second argument let $Y \stackrel{g}{\longrightarrow} Y'$. (As the restriction to the subspace $\sigma$-algebra, $\sa_{X \times Y} \subseteq \sa_{X \otimes Y}$, should now be evident, it is not included in the following diagrams.)  
This gives us the commutative square
 \begin{equation} \nonumber
 \begin{tikzpicture}[baseline=(current bounding box.center)]
 	\node	(XTYp)	at	(0,2)		         {$X \otimes \G(Y')$};
	\node	(TXYp)	at	(6,2)	                  {$\G(X \otimes Y')$};
	\node	(XTY)	at	(0,0)	                  {$X \otimes \G(Y)$};
	\node        (TXY)        at       (6,0)                  {$\G(X \otimes Y)$};
	
	\draw[->, right,auto] (XTYp) to node  {$\tau''_{X,Y'}$} (TXYp);
	\draw[->, above, auto] (XTY) to node {$1_X \otimes \G(g)$} (XTYp);
	\draw[->,right,auto]  (XTY)  to node {$\tau''_{X,Y}$}  (TXY);
	\draw[->,right] (TXY) to node {$\G(1_X \otimes g)$} (TXYp);
%	\draw (7.4,-1.1) node {$ev$};                    % auto placement of label sucks.	
 \end{tikzpicture}
 \end{equation}
 which maps elements according to
  \begin{equation} \nonumber
 \begin{tikzpicture}[baseline=(current bounding box.center)]
 	\node	(XpTY)	at	(0,2)		         {$(x, Qg^{-1})$};
	\node	(TXYp)	at	(6,2)	                  {$Qg^{-1} \tilde{\Gamma}_{x}^{-1}=Q \Gamma_{x}^{-1} (1 \otimes g)^{-1} $};
	\node	(XTY)	at	(0,0)	                  {$(x,Q)$};
	\node        (TXY)        at       (6,0)                  {$Q \Gamma_x^{-1}$};
	\draw[|->, right,auto] (XTYp) to node  { } (TXYp);
	\draw[|->, above, auto] (XTY) to node { } (XTYp);
	\draw[|->,right,auto]  (XTY)  to node { }  (TXY);
	\draw[|->,above,auto] (TXY) to node { } (TXYp);

 \end{tikzpicture}
 \end{equation}
where equality in the upper right hand corner follows from the observation that

  \begin{equation} \nonumber
 \begin{tikzpicture}[baseline=(current bounding box.center)]
 	\node	(Y)	at	(0,0)		         {$Y$};
	\node	(Yp)	at	(0,2)	         {$Y'$};
	\node	(XYp)	at	(2,2)	                  {$X \otimes Y'$};
	\node        (XY)        at       (2,0)                  {$X \otimes Y$};
	\draw[->, above] (Yp) to node  {$\tilde{\Gamma}_x$} (XYp);
	\draw[->, left] (Y) to node {$g$} (Yp);
	\draw[->,below]  (Y)  to node {$\Gamma_x$}  (XY);
	\draw[->,right] (XY) to node {$1_X \otimes g$} (XYp);

 \end{tikzpicture}
 \end{equation}

\end{proof}

\vspace{.1in}

Using the fact that the monoidal structure is symmetric we also have a natural transformation defined on components by
\begin{equation}   \nonumber
 \begin{tikzpicture}[baseline=(current bounding box.center)]
         \node       (XY)    at      (0,0)          {$\G(X) \otimes Y$};
 	\node	(XY2)	at	(5,0)     {$\G(X \times Y)$};
	\node       (y)      at      (0,-.8)          {$(P,y)$};
 	\node	(xy)	at	(5,-.8)       {$P\Gamma_y^{-1}$};

	\draw[->,above] (XY) to node {$\tau'_{X,Y}$} (XY2);
	\draw[|->] (y) to node {} (xy);

 \end{tikzpicture}.
 \end{equation}
The superscript on $\tau$ is used to denote which coordinate the probability measures are given.

The natural transformation $st$
\begin{equation}   \nonumber
 \begin{tikzpicture}[baseline=(current bounding box.center)]
         \node       (XY)    at      (0,0)          {$X^Y$};
 	\node	(XY2)	at	(5,0)     {$\G(X)^{\G(Y)}$};
	\node       (y)      at      (0,-.8)          {$g$};
 	\node	(xy)	at	(5,-.8)       {$\G(g)$};

	\draw[->,above] (XY) to node {$st_{X,Y}$} (XY2);
	\draw[|->] (y) to node {} (xy);

 \end{tikzpicture}
 \end{equation}
can now be constructed using the natural transformation $\tau''$ via
\begin{equation}   \nonumber
 \begin{tikzpicture}[baseline=(current bounding box.center)]
         \node       (XY)    at      (0,0)          {$Y^X$};
 	\node	(XY2)	at	(5,0)     {$\G(Y)^{\G(X)}$};
	\node       (LL)      at      (0,-2)          {$(Y^X \otimes \G(X))^{\G(X)}$};
 	\node	(LR)	at	(5,-2)       {$\G(Y^X \otimes X)^{\G(X)}$};

	\draw[->,above] (XY) to node {$st_{X,Y}$} (XY2);
	\draw[->,left] (XY) to node {$\Gamma_{Y^X}$} (LL);
	\draw[->,below] (LL) to node {${\tau''_{Y^X,X}}^{\T(X)}$} (LR);
	\draw[->,right] (LR) to node {$\G(ev_{Y^X,X})^{\G(X)}$} (XY2);
	
         \node       (XY2)    at      (8.5,0)          {$f$};
 	\node	(XY3)	at	(12,0)     {$\G(f)=\G(ev_{Y^X,X}) \circ \tau''_{Y^X,X} \circ \Gamma_f$};
	\node       (LL2)      at      (8.5,-2)          {$\Gamma_f$};
 	\node	(LR2)	at	(12,-2)       {$\tau''_{Y^X,X} \circ \Gamma_f$};

	\draw[|->,above] (XY2) to node {} (XY3);
	\draw[|->,left] (XY2) to node {} (LL2);
	\draw[|->,below] (LL2) to node {} (LR2);
	\draw[|->,right] (LR2) to node {} (XY3);

 \end{tikzpicture}
 \end{equation}
where $\Gamma_{Y^X}$ is the unit of the adjunction $\bullet \otimes \G(X) \vdash \bullet^{\G(X)}$, 
\begin{equation}   \nonumber
 \begin{tikzpicture}[baseline=(current bounding box.center)]
         \node       (Y)    at      (0,0)          {$Y^X$};
 	\node	(XY)	at	(5,0)		         {$(Y^X \otimes \G(X))^{\G(X)}$};
	         \node       (y)    at      (0,-.8)          {$f$};
 	\node	(xy)	at	(5,-.8)		         {$\Gamma_f$};

	\draw[->,above] (Y) to node {$\ulcorner \Gamma_Y \urcorner$} (XY);
	\draw[|->] (y) to node {} (xy);

 \end{tikzpicture}
 \end{equation}
with $\Gamma_{f}$ is the constant graph function with value $f$.  The equality 
\be \nonumber
\G(f) = \G(ev_{X,Y}) \circ \tau''_{Y^X,X} \circ \Gamma_{f}
\ee
then follows from 
\be \nonumber
\begin{array}{lcl}
\underbrace{\G(f)(\ulcorner P \urcorner)}_{=Pf^{-1}} &=& \Bigg(\G(ev_{X,Y}) \circ \tau''_{Y^X,X} \circ \Gamma_{f}\Bigg)(\ulcorner P \urcorner) \\
&=&  \G(ev_{Y^X,X}) (\tau''_{Y^X,X} (f,\ulcorner P \urcorner))  \\
&=& \G(ev_{Y^X,X}) (P\tilde{\Gamma}_f^{-1}) \quad \textrm{ note }\Gamma_f \ne \tilde{\Gamma}_f\\
&=&P\tilde{\Gamma}_f^{-1} ev_{Y^X,X}^{-1} \\
&=& Pf^{-1} 
\end{array}
\ee
where the last line follows from the fact the constant graph map $\tilde{\Gamma}_f$ satisfies the equation expressed by the commutativity of the diagram 
\begin{equation}   \nonumber
 \begin{tikzpicture}[baseline=(current bounding box.center)]
         \node       (X)    at      (0,0)          {$X$};
 	\node	(XY)	at	(3,0)		         {$Y^X \otimes X$};
	         \node       (Y)    at      (3,-2)          {$Y$};
 
	\draw[->,above] (X) to node {$\tilde{\Gamma}_f$} (XY);
	\draw[->,right] (XY) to node {$ev_{X,Y}$} (Y);
	\draw[->,left] (X) to node [xshift=-5pt,yshift=-3pt]{$f$} (Y);

 \end{tikzpicture}
 \end{equation}

\section{The construction of Bayesian inference maps}  
The construction of Bayesian inference maps using the symmetric monoidal closed category structure of $\M$ applies to both the deterministic and nondeterministic Bayesian models.  For simplicity  we first present the deterministic case  and subsequently show the minor changes required to address the nondeterministic Bayesian model.

However, it is worth noting that in constructing Bayesian inference maps it  suffices (for standard spaces) to  consider the special case where the given conditional probability $P_{Y|X}$ is a deterministic probability.   The general problem, as given in diagram \ref{BayesDiagram}, can then be solved using the same approach on the Bayesian model $(P_{X \times Y}, \delta_{\pi_Y})$, expressed by the  $K(\G)$-diagram   

  \begin{equation}  
 \begin{tikzpicture}[baseline=(current bounding box.center)]
 	\node	(1)	at	(0,0)		         {$1$};
	\node	(Y)	at	(3,-4)	               {$Y$};
	\node	(XY)	at	(-3,-4)	               {$X \times Y$};
	\node        (XYY)  at      (0,-2)               {$X \times Y \times Y$};
%	\node        (c)     at      (6.,-.9)         {in $K(\T)$};
	
	\draw[->,out=0,in=90,right] (1) to node  {$P_Y =P_{X \times Y}\pi_Y^{-1}$} (Y);
	\draw[->,out=-180,in=90,left] (1) to node  {$P_{X \times Y}$} (XY);
	\draw[->,left] (1) to node {$P_{X \times Y}\Gamma_{\pi_Y}^{-1}$} (XYY);
	\draw[->,above] (XYY) to node {$\delta_{\pi_{Y}}$} (Y);
	\draw[->,above] (XY) to node [xshift=-2pt] {$\delta_{\Gamma_{\pi_{Y}}}$} (XYY);
	
	\draw[->, above] ([yshift=2pt]XY.east) to node {${\delta_{\pi_Y}}$} ([yshift=2pt]Y.west);
	\draw[->, below,dashed] ([yshift=-2pt]Y.west) to node {$P_{X \times Y|Y}$} ([yshift=-2pt]XY.east);

 \end{tikzpicture}
 \end{equation}
to construct the conditional probability $P_{X \times Y|Y}$, which in turn specifies $P_{X|Y}$.

\subsection{Inference maps for deterministic  models}  
Our generic deterministic Bayesian problem (base model)  is represented by the diagram
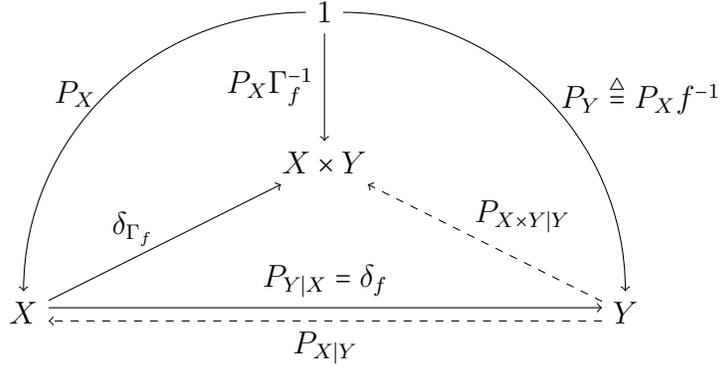
\begin{figure}[H]
  \begin{equation}   \nonumber
 \begin{tikzpicture}[baseline=(current bounding box.center)]
 	\node	(1)	at	(0,0)		         {$1$};
	\node	(Xf)	at	(4,-4)	               {$Y$};
	\node	(X)	at	(-4,-4)	               {$X$};
	\node        (XXf)  at      (0,-2)               {$X \times Y$};
	
	\draw[->,out=0,in=90,right] (1) to node  [xshift=3pt]{$P_Y \stackrel{\triangle}{=} P_Xf^{-1}$} (Xf);
	\draw[->,out=180,in=90,left] (1) to node  {$P_X$} (X);
	\draw[->,left] (1) to node {$P_X\Gamma_{f}^{-1}$} (XXf);
	\draw[->,above,right,dashed] (Xf) to node [xshift=-8pt,yshift=10pt]{$P_{X \times Y | Y}$} (XXf);
	\draw[->,above,left] (X) to node  [xshift=0pt,yshift=5pt]{$\delta_{\Gamma_{f}}$} (XXf);
	\draw[->,above] ([yshift=2pt] X.east) to node {$P_{Y|X}=\delta_{f}$} ([yshift=2pt] Xf.west);
	\draw[->, below,dashed] ([yshift=-3pt] Xf.west) to node {$P_{X|Y}$} ([yshift=-3pt] X.east);

 \end{tikzpicture}
 \end{equation}
\caption{The generic deterministic Bayesian model.}
\label{baseModelDiagram}
 \end{figure}
 \noindent
where the dashed arrows are the inference maps to be determined.  

The inference maps can be constructed by  taking the pullback of the $\M$-diagram\footnote{The notation $\ulcorner P_X \Gamma_f^{-1} \urcorner$ is used to emphasize that the diagram is in $\M$, rather than $K(\G)$ where probabilities are actual arrows.}
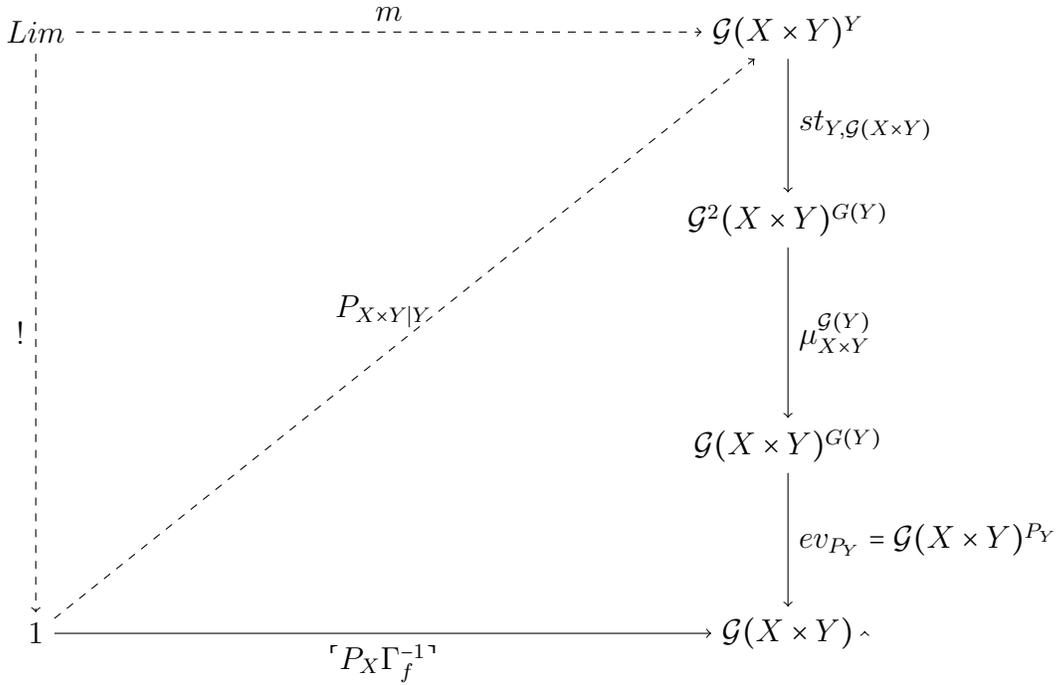
\begin{equation}   \nonumber
 \begin{tikzpicture}[baseline=(current bounding box.center)]
         \node       (R)    at      (0,-2)          {$Lim$};
 	\node	(1)	at	(0,-10)        {$1$};
	\node        (GXY) at     (10,-10)           {$\G(X \times Y)$};
%	\node        (GY2)   at      (10,0)          {$(\G(X) \otimes Y)^Y$};
	\node       (GXYY)  at    (10,-2)         {$\G(X \times Y)^Y$};
%	\node       (GXY2)  at    (10,-4)         {$\G(\G(X \times Y)^Y)$};	
	\node        (Gst)    at    (10,-4.5)        {$\G^2(X \times Y)^{G(Y)}$};
	\node        (G)     at      (10,-7.5)         {$\G(X \times Y)^{G(Y)}$};
%	\node        (G2)     at      (10,-10)         {$\G(X \times Y)^{G(Y)}$};

	\draw[->,below] (1) to node {$\ulcorner P_X\Gamma_f^{-1} \urcorner$} (GXY);
	
%	\draw[->,right] (GY2) to node  {$\tau'^Y$} (GXYY);         %
	\draw[->,above,dashed] (R) to node {$m$} (GXYY);
	\draw[->,left,dashed] (R) to node {$!$} (1);
	\draw[->,right] (GXY) to node {$ $} (GXY);
%	\draw[->,right] (GXYY) to node {$\eta_{\G(X \times Y)^Y}$} (GXY2);
	\draw[->,right] (GXYY) to node {$st_{Y,\G(X \times Y)}$}(Gst);
	\draw[->,right] (Gst) to node {$\mu_{X \times Y}^{\G(Y)}$} (G);
%	\draw[->,right] (G) to node {$\G(id_{X \times Y})$} (G2);
	\draw[->,right] (G) to node {$ev_{P_Y}=\G(X \times Y)^{P_Y}$}(GXY);
	\draw[->,dashed,above] (1) to node [xshift=-8pt]{$P_{X \times Y|Y}$} (GXYY);

 \end{tikzpicture}
 \end{equation}
\noindent
%The pullback map $m$ is clearly monic because both $\ulcorner P_X\Gamma_f^{-1} \urcorner$ is monic. 
where the map $ev_{P_Y}$ is evaluation at $P_Y$, i.e, given any map $\ulcorner T \urcorner \in \G(X \times Y)^{\G(Y)}$ precompose it with $\ulcorner P_Y \urcorner$.
The map $m$ in this pullback diagram is the subset of all $P_{X \times Y|Y} \in \G(X \times Y)^Y$, which upon application of the vertical arrows on the right hand side of the diagram, then yield the composite 

\begin{equation}   \nonumber
 \begin{tikzpicture}[baseline=(current bounding box.center)]
         \node       (1)    at      (0,0)          {$1$};
 	\node	(XY)	at	(4,0)		         {$\G(Y)$};
	 \node       (y)    at      (4,-2)          {$\G^2(X \times Y)$};
 	\node	(xy)	at	(0,-2)		         {$\G(X \times Y)$};

	\draw[->,above] (1) to node {$\ulcorner P_Y \urcorner$} (XY);
	\draw[->,right] (XY) to node {$\G(P_{X \times Y|Y})$} (y);
	\draw[->,below] (y) to node {$\mu_{X \times Y}$} (xy);
	\draw[->,left] (1) to node {$\ulcorner P_X\Gamma_f^{-1} \urcorner$} (xy);

 \end{tikzpicture}
 \end{equation}
which coincides with $P_X \Gamma_f^{-1}$ since the diagram is a pullback.  The commutativity of this diagram is Bayes equation for the deterministic model.  The fact the diagram is a pullback is trivial as $\M$ has all limits and the point $\ulcorner P_X \Gamma_f^{-1} \urcorner$ is monic so the limit corresponds to a subset of $\G(X \times Y)^Y$.
  
 The limit, $Lim$, in the pullback can be constructed, and characterized, using  the tensor product probability $P_X \otimes P_Y$, or alternatively, it can be constructed ``pointwise'' using the family of tensor probabilities $P_X \otimes \delta_y = P_X \Gamma_y^{-1}$.\footnote{In the standard theoretical approach to computing inference maps the Radon Nikodym (RN) derivatives are computed with respect to measures on the  space $Y$ rather than the product space $X \times Y$.  Using RN derivatives, either directly or indirectly, on the product space is  preferable as subsequent arguments illustrate. (One need not ``glue together'' a family of RN derivatives.)}
 
 \textbf{Method 1: Using Radon Nikodym derivatives.} Note that $P_X \Gamma_f^{-1} \ll P_X \otimes P_Y$. On the basic rectangles, $A \times B \in \sa_{X \times Y}$ we have $P_X\Gamma_f^{-1}(A \times B) = P_X(A \cap f^{-1}(B)) \le \min(P_X(A),P_Y(B))$.\footnote{The probability $P_X\Gamma_f^{-1}$ is the push forward of the tensor product probability  $P_X \otimes P_Y$ by the idempotent map $\Gamma_f \circ \pi_X$,  $(P_X \otimes P_Y)(\Gamma_f \circ \pi_X)^{-1} = P_X\Gamma_f^{-1}$.}  Let $h$ be a Radon Nikodym derivative for this absolute continuity condition,
\be \label{RN}
\begin{array}{lcl}
P_X\Gamma_f^{-1}(\zeta) &=& \int_{\zeta} h \, d(P_X \otimes P_Y) \\
&=& \int_Y \int_{\zeta} h \, d(\underbrace{P_X\Gamma_y^{-1}}_{=P_X \otimes \delta_y}) \, dP_Y \\
&=& \int_Y (\int_{\Gamma_y^{-1}(\zeta)} h(\cdot,y) \, dP_X) dP_Y
\end{array}
\ee
As $P_X\Gamma_f^{-1}(X \times Y)=1$ it follows that $\int_X h(\cdot, y) \, dP_X=1$ $P_Y-a.e.$.  Suppose $V \in \sa_Y$ such that $\int_X h(\cdot, y) \, dP_X=1$ for all $y \in V$.

Then 
\be \nonumber
P_{X \times Y|Y}(\zeta | y)  \stackrel{\triangle}{=} \left\{ \begin{array}{ll} \int_{\zeta} h(\cdot, \cdot) \, d(P_X \Gamma_y^{-1}) & y \in V\\ 
P_X\Gamma_f^{-1} & y \not \in V \end{array} \right.
\ee
defines a conditional probability, the inference map for the deterministic model,  which  
is absolutely continuous with respect to $P_X \Gamma_y^{-1}$, and by equation (\ref{RN}), the weighted sum
\be \label{integrate}
P_X \Gamma_f^{-1} = \int_Y P_{X \times Y}(\cdot |y) \, dP_Y.
\ee   
The sequence of vertical arrows in the pullback diagram show that this expression make sense \emph{without requiring evaluation} on any measurable set in $X \times Y$, as the multiplication natural transformation $\mu_{X \times Y}$ and $ev_{P_Y}$ allows us to ``integrate'' without evaluation on any measurable set.\footnote{This statement requires the SMCC structure where the evaluation maps 
\be \nonumber
\G(X \times Y)^{\G(Y)} \otimes \G(Y) \stackrel{ev}{\longrightarrow} \G(X \times Y)
\ee
are measurable.  Formally, equation (\ref{integrate}) is the messy expression
\be \nonumber
(ev_{P_X} \circ \mu_{X \times Y}^{\G(Y)} \circ st_{Y,\G(X \times Y)} \circ \eta_{\G(X \times Y)^Y})(P_{X \times Y|Y}) = P_X\Gamma_f^{-1}.
\ee 
We feel justified in using the integral sign, because viewing probability measures as weakly averaging affine functionals which preserve limits,  integrals $\int$ correspond precisely to evaluation maps.\cite{Sturtz} }

\vspace{.1in}

\textbf{Method  2: Using the constant graph maps.}
 For each point $y \in Y$,  consider the $K(\G)$-diagram
 \begin{equation}   \nonumber
 \begin{tikzpicture}[baseline=(current bounding box.center)]
 	\node	(1)	at	(0,0)		         {$1$};
	\node	(Y)	at	(3,-4)	               {$\{y\} \cong 1$};
	\node	(X)	at	(-3,-4)	               {$X$};
	\node        (XY)  at      (0,-2)               {$X \times Y$};
	
%	\draw[->,out=0,in=90,right] (1) to node  {$P_Y$} (Y);
	\draw[->,out=180,in=90,left] (1) to node  {$P_X$} (X);
	\draw[->,left] ([xshift=-3pt]1.south) to node {$P_X \Gamma_f^{-1}$} ([xshift=-3pt]XY.north);
	\draw[->,right] ([xshift=3pt]1.south) to node {$P_X\Gamma_y^{-1}$} ([xshift=3pt]XY.north);
	\draw[->,above,dashed] (Y) to node [xshift=40pt]{$P_{X \times Y|Y}(\cdot | y) = P_{a}^y$} (XY);
%	\draw[->,above] (X) to node  [xshift=-8pt,yshift=0pt]{$P_{X \times Y | X}$} (XY);
	
	\draw[->,above] ([yshift=2pt] X.east) to node [xshift=-5pt]{$\delta_{\Gamma_f}$} ([yshift=2pt] XY.west);
	\draw[->, below] ([yshift=-3pt] X.east) to node [xshift=5pt]{$\delta_{\Gamma_y}$} ([yshift=-3pt] XY.west);

 \end{tikzpicture}
 \end{equation}
Provided that the two measures $P_X\Gamma_y^{-1}, P_X\Gamma_f^{-1} \in \G(X \times Y)$ are not singular with respect to each other,  \mbox{$P_X \Gamma_y^{-1} \not \perp P_X \Gamma_f^{-1}$},  Lebesques Decomposition theorem allows us to uniquely write $P_X \Gamma_f^{-1}$ as a convex sum of two probability measures, 
\be \nonumber
P_X \Gamma_f^{-1} = \alpha P_{a}^y + (1-\alpha) P_{s}^y.
\ee

If we define\footnote{It would be more elegant to define
\be \nonumber
P_{X \times Y|Y}( \cdot | y) = \left\{ \begin{array}{ll} P_a^{y}   & \textrm{ whenever }P_X \Gamma_f^{-1} \not \perp P_X \Gamma_y^{-1} \\
P_X \Gamma_f^{-1} & \textrm{otherwise}
\end{array} \right.
\ee    However, we do not know how to prove the set
$\{y \in Y \, | \, P_X \Gamma_f^{-1}  \perp P_X \Gamma_y^{-1} \}$
is measurable.}

\be \nonumber
P_{X \times Y|Y}( \cdot | y) = \left\{ \begin{array}{ll} P_a^{y}   & y \in V \\
P_X \Gamma_f^{-1} & y \not \in V
\end{array} \right.
\ee
then the two characterizations are equivalent.
The limit of the pullback is the subset  $Lim \subset \G(X \times Y)^Y$ consisting of all possible inference maps.  However, these inference maps   are $P_Y-a.e.$ equal, which is just the statement in equation (\ref{integrate}).  The following results are now obvious.

\begin{lemma} The function 
\begin{equation}   \nonumber
 \begin{tikzpicture}[baseline=(current bounding box.center)]
         \node       (Y)    at      (0,0)          {$Y$};
 	\node	(XY)	at	(8,0)		         {$\G(X\times Y)$};
	 \node       (y)    at      (0,-.8)          {$y$};
 	\node	(xy)	at	(8,-.8)		         {$P_{X \times Y|Y}(\bullet | y)=\int_{\bullet} h \, d(P_X \Gamma_y^{-1})$};

	\draw[->,above] (Y) to node {$P_{X\times Y |Y}$} (XY);
	\draw[|->] (y) to node {} (xy);

 \end{tikzpicture}
 \end{equation}
 is measurable.
 \end{lemma}

\begin{thm} The conditional probability $P_{X \times Y|Y}$ satisfies the condition 
\be \nonumber
P_{X \times Y|X} \star P_X = P_X\Gamma_f^{-1} = P_{X \times Y|Y} \star P_Y
\ee
making $P_{X \times Y|Y}$ and  $P_{X|Y}= \delta_{\pi_X} \star P_{X \times Y|Y}$ inference maps for the deterministic  Bayesian model  $(P_X, \delta_f)$.
\end{thm}
\begin{proof}
This follows directly from equation (\ref{RN}) as $P_{X \times Y|Y}(\zeta | y) = \int_{\zeta} h \, d(P_X \Gamma_y^{-1})$.
\end{proof}
   
\subsection{Inference maps for nondeterministic models}
The only change required from the deterministic construction is to replace the probability measure $P_X\Gamma_f^{-1}: 1 \rightarrow \G(X \times Y)$ 
used in the pullback construction with the joint probability measure specified by the composite

\begin{equation}   \nonumber
 \begin{tikzpicture}[baseline=(current bounding box.center)]

 	\node	(1)	at	(0,0)		         {$1$};
	\node	(GX)	at	(2,0)	               {$\G(X)$};
	\node        (GXGY)  at   (6,0)          {$\G(X \otimes \G(Y))$};
	\node        (GGXY)    at   (10,0)        {$\G(\G(X \otimes Y))$};
	\node        (GXY) at     (14,0)           {$\G(X \otimes Y)$};
	\node        (F)    at     (14,-2)         {$\G(X \times Y)$};

	\draw[->,above] (1) to node {$\ulcorner P_X\urcorner$} (GX);
	\draw[->,above] (GX) to node {$\G(\Gamma_{P_{Y|X}})$} (GXGY);
	\draw[->,above] (GXGY) to node  {$\G(\tau''_{X,Y})$} (GGXY);
	\draw[->,above] (GGXY) to node {$\mu_{X \times Y}$}(GXY);
	\draw[->,left] (GXY) to node {$\G(id_{X\times Y})$} (F);

 \end{tikzpicture}
 \end{equation}

\section{Appendix} The inadequacy of kernels.

\paragraph{Example}
 Let $\mathcal{L}$ denote the Lebesque measure.    Take the identity map $[0,1] \rightarrow [0,1]$ where the domain space  has the $\sigma$-algebra of Lebesque measurable sets, and the codomain space has the Borel $\sigma$-algebra. 
  The deterministic Bayesian model $( \mathcal{L},\delta_{id})$ apparently has no corresponding inference map.  The logical choice for the inference map,  the identity map $\delta_{id_{[0,1]}}$ is not an option because for any non Borel measurable set $F$ the map $\delta_{id_{[0,1]}}(F|\cdot)$ is not measurable.
 
 \vspace{.1in}
 
 This elementary problem captures the difficulty arising in constructing inference maps.  Namely, problems can arise on sets of measure $0$.  Restricting $\M$ to  Standard Spaces or Polish Spaces circumvents the problem by not allowing the use of Lebesque measurable sets,  or any other ``unpleasant'' $\sigma$-algebra where the same type of difficulty with sets of measure $0$ arises.  
 
On the other hand, this problem can be avoided if instead of using kernels, which is what we have referred to as (regular) conditional probabilities,  we use $\mathcal{D}$-kernels.  

Suppose $f: X \rightarrow Y$, and $P_X \in \G(X)$.  If $\nu: \sa_Y \times X \rightarrow [0,1]$ then we say $\nu$ is a  $\mathcal{D}$-kernel provided 
\begin{enumerate}
 \item for each $x \in X$, $Q( \cdot | x)$ is a probability measure on $Y$, and
 \item for each $A \in \sa_X$  $Q(A |\cdot)$ is a $\sa_Y^* - \B$ measurable function, 
 where $\sa_X^{*}$ is the completion of $\sa_X$ with respect to $P_X$.
 \end{enumerate}
 The difference between kernels and $\mathcal{D}$-kernels is the kernels requires measurability with respect to $\sa_X$ rather than the completion with respect to the measure $P_X$.  
 
 Using $\mathcal{D}$-kernels the identity map $id_{[0,1]}$ is an inference map for the example.  So why not use these instead of kernels? 
 The difficulty with using $\mathcal{D}$-kernels arises from the awkward category it imposes upon us due to the fact we need to be able to compose such maps.  
 
For example, we could take the objects to be a triple $(X, \sa_X,  \mathbf{D})$ where $\mathbf{D} \subseteq \G(X)$, and  an arrow $Q$
  \begin{equation}   \nonumber
 \begin{tikzpicture}[baseline=(current bounding box.center)]
	
	\node	(X)	at	(0,0)	               {$(X, \sa_X,  \mathbf{D})$};
	\node        (Y)  at      (4,0)               {$(Y, \sa_Y,  \mathbf{E})$};

	\draw[->,above] (X) to node [xshift=0pt]{$Q$} (Y);
 \end{tikzpicture}
 \end{equation}
is a $\mathcal{D}$-kernel for every $P_X \in \mathbf{D}$, and $Q( \cdot | x) \in \mathbf{E}$ for every $x \in X$.  The set $\mathbf{E}$  plays no role whatsoever in the requirement of the arrow $Q$. But upon precomposition with any other arrow   

  \begin{equation}   \nonumber
 \begin{tikzpicture}[baseline=(current bounding box.center)]
	
	\node	(X)	at	(0,0)	               {$(X, \sa_X,  \mathbf{D})$};
	\node        (Y)  at      (4,0)               {$(Y, \sa_Y,  \mathbf{E})$};
	\node        (W)  at      (-4,0)               {$(W, \sa_W,  \mathbf{F})$};

	\draw[->,above] (X) to node [xshift=0pt]{$Q$} (Y);
	\draw[->,above] (W) to node [xshift=0pt]{$R$} (X);
 \end{tikzpicture}
 \end{equation}
 the  composition requires the integrand in
 \be \nonumber
\int_X Q(\zeta | \cdot) \, dR(\cdot | w)
\ee 
to be $R(\cdot | w)^*$-measurable.  Hence the requirement that each $R(\cdot|w) \in \mathbf{D}$ for every $w \in W$.  Changing $R$ to another arrow shows that  
for such a scheme to work in defining a category requires that with every space $(X, \sa_X)$ in $\M$, we need to make a lot of copies of it,  so we can proceed to have every possible composition we had within the $K(\G)$ framework using kernels.

 \begin{flushleft}
Kirk Sturtz \\
Universal Mathematics \\
%Dayton, OH USA \\
\email{kirksturtz@UniversalMath.com} \\
\end{flushleft}

 \end{document}